\documentclass[11pt]{article}
\usepackage[pdftex,marginratio=1:1,width=180mm,height=235mm,top=20mm,headsep=15mm,head=10mm,footskip=10mm]{geometry}
\usepackage{amsmath,amsfonts,amssymb,amsthm,cite,graphicx,caption,subcaption,bm,array,booktabs,multirow,algorithm,authblk,lineno,xcolor,lipsum,epstopdf,color}
\usepackage[english]{babel}
\usepackage{mathtools}
 
\title{Novel Results of Two Generalized Classes of Fibonacci and Lucas Polynomials and Their Uses in the Reduction of Some Radicals}
\author{W.M. Abd-Elhameed 1\thanks{walee$_{-}$9@yahoo.com} }
\author{N.A. Zeyada 1\thanks{nzeyada@gmail.com}}
\author{A.N. Philippou 2\thanks{professoranphilippou@gmail.com}}
\affil{\small{1 Department of Mathematics, Faculty of Science, Cairo
University, Giza 12613, Egypt}} \affil{\small{2 Department of
Mathematics, University of Patras, Patras 26504, Greece}}
 \date{}

\newcommand{\p }{\\[-0.7cm]}

\newcommand{\dd }{\displaystyle}

\newcommand{\bq }{\begin{equation}}
\newcommand{\eq }{\end{equation}}

\theoremstyle{plain}
\newtheorem{thm}{Theorem}
\newtheorem{lem}{Lemma}

\newtheorem{cor}{Corollary}
\newtheorem{rem}{Remark}

\newtheorem{exa}{Example}
\begin{document}
\maketitle
\begin{abstract}
This paper is concerned with developing some new connection formulae
between two generalized classes of Fibonacci and Lucas polynomials.
All the connection coefficients involve hypergeometric functions of
the type $_2F_{1}(z)$, for certain $z$. Several new connection
formulae between some famous polynomials such as Fibonacci, Lucas,
Pell, Fermat, Pell-Lucas, and Fermat-Lucas polynomials are deduced
as special cases of the derived connection formulae. Some of the
introduced formulae generalize some of those existing in the literature.
As two applications of the derived connection formulae, some new
formulae linking some celebrated numbers are given and also some
newly closed formulae of certain definite weighted integrals are
deduced. Based on using the two generalized classes of Fibonacci and
Lucas polynomials, some new reduction formulae of certain odd and
even radicals are developed.
\end{abstract}
\smallskip



%%%%%%%%%%%%%%%%%%%%%%%%%%%%%%%%%%%%%%%%%%%%%%%%%%%%%%%%%%%%%%%%%%%%%%%%%%%%%%%%%%%%%%%%%%%%%%%%%%%%%%%%%%%%%%%%%
\section{Introduction}
 The Fibonacci and Lucas sequences are of the most important
number sequences. These numbers and their corresponding polynomials
can be generated by recurrence relations each of degree two. In
fact, these sequences of polynomials and numbers are crucial in
different branches such as number theory, combinatorics, numerical
analysis, probability, biology and physics, so investigations of
these sequences attracted the attention of many mathematicians and
scientists. From a theoretical point of view, there are several
studies concerning Fibonacci and Lucas sequences and their
generalizations. Some of these studies can be found in
\cite{zhang2002chebyshev,gulec2013new,falcon2007fibonacci,abd2017generalization
,dilcher2000hypergeometric,Ind,dafnis2018identity,ye2017common}).
From a practical point of view, some recent articles employ
Fibonacci and Lucas polynomials and their generalized polynomials in
treating some kinds of ordinary and fractional differential
equations. For example, Abd-Elhameed and Youssri in
\cite{abd2017generalized,abd2016spectral,abd2016novel,abdspectral}
have developed some numerical algorithms for solving some kinds of
fractional differential equations based on employing these kinds of
polynomials.
\smallskip
\par
If we have two polynomial sets $\{A_{i}(x)\}_{i\ge 0}$ and $\{B_{j}(x)\}_{j\ge 0}$, then to solve the connection problem between them, we have to find the connection coefficients $C_{i,j}$ in the
equation
\[A_{i}(x)=\sum_{j=0}^{i}C_{i,j}\, B_{j}(x).\]
 The coefficients $C_{i,j}$ have prominent parts in several problems in mathematics as well as in mathematical physics. Due to this importance, the connection problems between various polynomials have been investigated by many authors. In this regard, Abd-Elhameed et al. in \cite{RamFib} solved the connection problems between Fibonacci polynomials and Chebyshev polynomials of first and second kinds. Some other studies concerning connection problems can be found in \cite{area2001solving,maroni2008connection}.
 \smallskip
\par
 The hypergeometric functions have vital parts in the area of special
 functions and their applications (see, for example \cite{andrews1999special}). In fact,  almost all of the elementary
 functions of mathematics are either hypergeometric or ratios of hypergeometric functions. In addition, connection coefficients and also linearization coefficients are often expressed in terms of hypergeometric functions of different types (see, for example \cite{Tcheutia,W-Single,abd2016linearization}).
 \smallskip
\par
 We can summarize the principal objectives of this article in the following points:
  \begin{itemize}
  \item
  Solving the connection problems between two certain classes of polynomials generalizing Fibonacci and Lucas polynomials. we will show that the obtained expressions involve hypergeometric functions of the type $_2F_{1}(z)$ for certain $z$.
  \item
  Developing some applications of the introduced connection formulae. In this respect, two applications are presented. In the first, some new formulae between some celebrated numbers are given. In the second, some definite weighted integrals are evaluated in closed forms.
   \item
  Employing the two classes of generalized Fibonacci and generalized Lucas polynomials to obtain new reduction formulae of certain kinds of even and odd radicals.
 \end{itemize}
\par
 The contents of the paper are organized as follows. The next section involves some fundamental properties and useful relations concerned with the two generalized Fibonacci and Lucas polynomials. Section 3 is concerned with developing new connection formulae between the two introduced generalized polynomials. We develop some other connection formulae between two polynomials belong to the same class of the generalized polynomials in Section 4. Section 5 is devoted to introducing two applications of the introduced connection formulae. Section 6 is interested in developing new expressions of certain even and odd radicals with the aid of the two introduced generalized classes of polynomials. New formulae of some other radicals are developed in Section 7. We end the paper by some conclusions in Section 8.

 \section{Some properties of two generalized classes of Fibonacci and Lucas polynomials}
 In this section, we give some properties concerned with two certain classes generalizing Fibonacci and Lucas polynomials.\\
  First, let $a,b,r,s$ be nonzero real numbers. Two classes generalizing Fibonacci and Lucas polynomials can be generated
  by means of the following two recurrence relations:
 \bq
  \label{recgenFib}
  \phi^{a,b}_{j}(x)=a\, x\, \phi^{a,b}_{j-1}(x)+b\, \phi^{a,b}_{j-2}(x),\quad \phi_{0}(x)=1,\  \phi_{1}(x)=a\, x,
 \eq
 and
  \bq
  \label{recgenLuc}
  \psi^{r,s}_{j}(x)=r\, x\, \psi^{r,s}_{j-1}(x)+s\, \psi^{r,s}_{j-2}(x),\quad \psi_{0}(x)=2,\  \psi_{1}(x)=r\, x.
 \eq
 Note that for each $j\ge 0$, $\phi^{a,b}_{j}(x)$ and $\psi^{r,s}_{j}(x)$ is of degree $j$.\\
 See also Philippou \cite{philippou1986distributions,philippouEncyclopedia1,philippouEncyclopedia2}, and references therein for
different generalized Fibonacci and Lucas polynomials, which, in addition to their importance
they are very useful in probability and reliability theory as well.\\
 The Binet's formulae of the two polynomials $\phi^{a,b}_{j}(x)$ and $\psi^{r,s}_{j}(x)$ are:
  \bq
  \label{BinetFib}
   \phi^{a,b}_{j}(x)=\dd\frac{\left(a\, x+\sqrt{a^2\, x^2+4\, b}\right)^{j+1}-\left(a\, x-\sqrt{a^2\, x^2+4\, b}\right)^{j+1}}{2^{j+1}\, \sqrt{a^2\, x^2+4\, b}},
  \eq
  and
    \bq
  \label{BinetLuc}
   \psi^{r,s}_{j}(x)=\dd\frac{\left(r\, x+\sqrt{r^2\, x^2+4\, s}\right)^j+\left(r\, x-\sqrt{r^2\, x^2+4\, s}\right)^j}{2^j}.
  \eq
    \smallskip
\par
Two of the most useful properties of the two polynomials $\phi^{a,b}_{j}(x)$ and $\psi^{r,s}_{j}(x)$ are their power form representations. These polynomials can be given explicitly in the following two combinatorial forms:
   \bq
   \label{powerphi}
   \phi^{a,b}_{j}(x)=\dd\sum_{m=0}^{\left\lfloor
 \frac{j}{2}\right\rfloor}\tbinom{j-m}{m}\, b^m\,  a^{j-2 m} \,  x^{j-2 m},
  \eq
   and
    \bq
   \label{powerpsi}
   \psi^{r,s}_{j}(x)=j\, \dd\sum_{m=0}^{\left\lfloor
 \frac{j}{2}\right\rfloor}\frac{s^m \, r^{j-2 m}\, \binom{j-m}{m} }{i-m}\,  x^{j-2 m},
   \eq
   The inversion formulas of \eqref{powerphi} and \eqref{powerpsi} are also of interest. They are given as follows:
      \bq
    \label{invGenFib}
   x^j =a^{-j}\, \dd\sum_{i=0}^{\left\lfloor
 \frac{j}{2}\right\rfloor} \dd\frac{(-b)^i\, \dd\tbinom{j}{i}\, (j-2i+1)}{j-i+1}\, \phi^{a,b}_{j-2i}(x),\quad j\ge 0,
   \eq
   and
   \bq
    \label{invGenLuk}
   x^j =r^{-j}\, \dd\sum_{i=0}^{\left\lfloor
 \frac{j}{2}\right\rfloor} (-s)^i\, c_{j-2i}\, \binom{j}{i}\, \psi^{r,s}_{j-2i}(x),\quad j\ge 0,
   \eq
where
\[c_{j}=\begin{cases}
 \frac{1}{2},&j=1,\\
 1,&j\ge 1.\end{cases}\]
  \smallskip
\par
The main advantage of using the two generalized polynomials $\phi^{a,b}_{j}(x)$ and $\psi^{r,s}_{j}(x)$ is that several important classes of polynomials can be obtained as special cases of them. In fact, Fibonacci, Pell, Fermat, and Chebyshev and Dickson polynomials of the second kind are special kinds of $\phi^{a,b}_{j}(x)$, while Lucas, Pell-Lucas, Fermat-Lucas, and Chebyshev and Dickson polynomials of the first kind are special polynomials of $\psi^{r,s}_{j}(x)$. Table 1 displays the different celebrated special classes of the two generalized classes of polynomials.

  \begin{table}[ht]
  \centering
  \caption{Special cases of the two generalizing classes}
\begin{tabular}{|l|l|l|l|}
 \hline
Fibonacci polynomials& $F_{j+1}(x)=\phi^{1,1}_{j}(x)$\\[0.5cm] \hline
Pell polynomials & $P_{j+1}(x)=\phi^{2,1}_{j}(x)$ \\[0.5cm]\hline
Fermat polynomials & $\mathcal{F}_{j+1}(x)=\phi^{3,-2}_{j}(x)$ \\[0.5cm]\hline
Chebyshev polynomials of second kind & $U_j(x)=\phi^{2,-1}_{j}(x)$ \\[0.5cm]\hline
Dickson polynomials of second kind & $E^{\alpha}_{j}(x)=\phi^{1,-\alpha}_{j}(x)$ \\[0.5cm]\hline
Lucas polynomials& $L_{j}(x)=\psi^{1,1}_{j}(x)$\\[0.5cm] \hline
Pell-Lucas polynomials & $Q_{j}(x)=\psi^{2,1}_{j}(x)$ \\[0.5cm]\hline
Fermat-Lucas polynomials & $f_{j}(x)=\psi^{3,-2}_{j}(x)$ \\[0.5cm]\hline
Chebyshev polynomials of first kind & $T_j(x)=\frac12\psi^{2,-1}_{j}(x)$ \\[0.5cm]\hline
Dickson polynomials of first kind & $D^{\alpha}_{j}(x)=\psi^{1,-\alpha}_{j}(x)$ \\[0.5cm]
  \hline
  \end{tabular}
\end{table}
\noindent
%The following two lemmas are useful in the sequel.
%\begin{lem} \cite{abd2017generalized,abdspectral}
%For every nonnegative integer $m$, the following inversion formulae of the polynomials $\phi^{a,b}_{j}(x)$ and $\psi^{a,b}_{j}(x)$ are valid:
%      \bq
%   x^m =\dd\frac{1}{a^m}\, \dd\sum_{i=0}^{\left\lfloor
% \frac{m}{2}\right\rfloor}\frac{(-1)^i\, \binom{m}{i}\, b^i\, (m-2 i+1)}{(m-i+1)}\, \phi^{a,b}_{m-2i}(x),
%   \eq
%   and
%   \bq
%   x^m =\dd\frac{1}{r^m}\, \dd\sum_{i=0}^{\left\lfloor
% \frac{m}{2}\right\rfloor} (-1)^i\, c_{m-2i}\, s^i\, \binom{m}{i}\, \psi^{r,s}_{m-2i}(x),
%   \eq
%   where $$c_{r}=\begin{cases}\frac12,&r=0,\\
%   1,&r\ge 1.\end{cases}$$
%   \end{lem}
   \section{Connection formulae between the two generalized classes of Fibonacci and Lucas polynomials}
  This section focuses on developing new connection formulae between the two generalized classes of polynomials which were introduced in Section 2. Several connection formulas between some famous polynomials are also deduced as special cases. The connection coefficients are expressed in $_{2}F_{1}(z)$ for certain $z$.\\
  Here, we recall the definition of the
hypergeometric function $_{2}F_{1}(z)$ (see \cite{andrews1999special})
\[_{2}F_{1}\left.\left(
\begin{array}{cccc}
 a_{1},a_{2}\\
 b_{1}\\
\end{array}
\right|z\right)=\dd\sum_{r=0}^{\infty}\dd\frac{(a_{1})_{r}\,
(a_{2})_{r}}{(b_{1})_{r}}\ \dd\frac{z^r}{r!} ,\]
 where
 $a_{1},\, a_{2}$ and $b_{1}$ are complex or real parameters, with
 $b_{1}\not=0$.\\
  Now, the following lemma is needed.
 \begin{lem}
 \label{hyperdiff}
 \label{lem1}
 Let $a,b,r,s$ be any nonzero real numbers and let
 \[A_{m,i}=\dd\frac{i\, s^m\, \left(\dd\frac{r}{a}\right)^{i-2 m}\, \binom{i-m}{m}}{i-m}\ _{2} F_{1}\left.\left(
\begin{array}{cc}
 -m,i-m\\
  i-2m+2
\end{array}
\right|\dd\frac{b\, r^2}{a^2\, s}\right).\]
 The following recurrence relation is satisfied by $A_{m,i}:$
  \bq
   \label{rec1}
  r\, A_{m,i-1}-r\, b\, A_{m-1,i-1}+a\, s\,
 A_{m-1,i-2}=a\, A_{m,i}.
  \eq
 \end{lem}
 \begin{proof}
 Lemma \ref{hyperdiff} can be proved via straightforward lengthy computations in view of the definition of the hypergeometric function $_2F_{1}(z)$.
 \end{proof}
    \begin{thm}
 \label{connecthm1}
  For every non-negative integer $i$, the following connection formula holds
 \bq
 \label{connect1}
\phi^{a,b}_{i}(x)=\left(\frac{a}{r}\right)^i\, \dd\sum_{m=0}^{\left\lfloor
 \frac{i}{2}\right\rfloor}c_{i-2 m}\, (-s)^m\, \binom{i}{m} \ _{2} F_{1}\left.\left(
\begin{array}{cc}
 -m,m-i\\
  -i
\end{array}
\right|\frac{b\, r^2}{a^2\, s}\right)\ \psi^{r,s}_{i-2m}(x).
 \eq
 \end{thm}

 \begin{thm}
 \label{connecthm2}
  For every non-negative integer $i$, the following connection formula holds
 \bq
 \label{connect2}
\psi^{r,s}_{i}(x)=i\, \dd\sum_{m=0}^{\left\lfloor
 \frac{i}{2}\right\rfloor}\frac{s^m\, \left(\dd\frac{r}{a}\right)^{i-2 m}\, \binom{i-m}{m}}{(i-m)(i-2m+1)}\ _{2} F_{1}\left.\left(
\begin{array}{cc}
 -m,i-m\\
  i-2m+2
\end{array}
\right|\dd\frac{b\, r^2}{a^2\, s}\right)\ \phi^{a,b}_{i-2m}(x).
 \eq
 \end{thm}
 \begin{proof}
  The proofs of Theorems \ref{connecthm1}, \ref{connecthm2} can be proved by similar procedures. So, it is sufficient to prove Theorem \ref{connecthm2}. We will proceed by induction. We assume that \eqref{connect2} holds for all $j<i$ and we will show that \eqref{connect2} is itself valid. If we start with the recurrence relation \eqref{recgenLuc} and apply the induction hypothesis twice, then we get
 \bq
 \label{ind21}
  \psi^{a,b}_{i}(x)=r\, x\, \dd\sum_{m=0}^{\left\lfloor
 \frac{i-1}{2}\right\rfloor} A_{m,i-1}\, \phi^{a,b}_{i-2m-1}(x)+s\, \dd\sum_{m=0}^{\left\lfloor
 \frac{i}{2}\right\rfloor-1} A_{m,i-2}\,\phi^{a,b}_{i-2m-2}(x),
 \eq
 where
 \[A_{m,i}=\frac{i\, s^m\, \left(\dd\frac{r}{a}\right)^{i-2 m}\, \binom{i-m}{m}}{i-m}\ _{2} F_{1}\left.\left(
\begin{array}{cc}
 -m,i-m\\
  i-2m+2
\end{array}
\right|\dd\frac{b\, r^2}{a^2\, s}\right).\]
 The recurrence relation \eqref{recgenFib} implies that
 \bq
 \label{recused}
 x\, \phi^{a,b}_{i-2m-1}(x)=\dd\frac{1}{a}\, \phi^{a,b}_{i-2m}(x)-\dd\frac{b}{a}\, \phi^{a,b}_{i-2m-2}(x).
 \eq
  If we insert \eqref{recused} into \eqref{ind21}, then the following relation is obtained
\[\begin{split}
  \psi^{a,b}_{i}(x)=&\dd\frac{r}{a}\dd\sum_{m=0}^{\left\lfloor
 \frac{i-1}{2}\right\rfloor} A_{m,i-1}\,  \phi^{a,b}_{i-2m}(x)-\dd\frac{r\, b}{a}\dd\sum_{m=0}^{\left\lfloor
 \frac{i-1}{2}\right\rfloor} A_{m,i-1}\, \phi^{a,b}_{i-2m-2}(x)\\
  &+s\, \dd\sum_{m=0}^{\left\lfloor
 \frac{i}{2}\right\rfloor-1} A_{m,i-2}\, \phi^{a,b}_{i-2m-2}(x).
 \end{split}
\]
  The last relation-after performing some algebraic computations- can be turned into
  the following equivalent one:
    \bq
    \begin{split}
 \label{ind23}
 \psi^{a,b}_{i}(x)=&\dd\sum_{m=1}^{\left\lfloor
 \frac{i-1}{2}\right\rfloor} \left(\frac{r}{a}A_{m,i-1}-\frac{r\, b}{a}A_{m-1,i-1}+s\,
 A_{m-1,i-2}\right)\, \phi^{a,b}_{i-2m}(x)+\frac{r}{a}\, A_{0,i-1}\, \phi^{a,b}_{i}(x)\\
  &-\frac{r\, b}{a}\, A_{\left\lfloor
 \frac{i-1}{2}\right\rfloor,i-1}\, \phi^{a,b}_{i-2\, \left\lfloor
 \frac{i-1}{2}\right\rfloor-2}(x)+s\, A_{\frac {i}{2}-1,i-2}\, \xi_{i},
 \end{split}
 \eq
where
 \[\xi_{i}=
  \begin{cases}
  1, &{i\ \text{even}},\\
  0,  &{i\ \text{odd}}.
  \end{cases}\]
 In virtue of Lemma \ref{lem1}, one can rewrite \eqref{ind23} in the form
     \bq
    \begin{split}
 \label{ind24}
 \psi^{a,b}_{i}(x)=&\dd\sum_{m=1}^{\left\lfloor
 \frac{i-1}{2}\right\rfloor} A_{m,i}\, \phi^{a,b}_{i-2m}(x)+\frac{r}{a}\, A_{0,i-1}\, \phi^{a,b}_{i}(x)\\
  &-\frac{r\, b}{a}\, A_{\left\lfloor
 \frac{i-1}{2}\right\rfloor,i-1}\, \phi^{a,b}_{i-2\, \left\lfloor
 \frac{i-1}{2}\right\rfloor-2}(x)+s\, A_{\frac {i}{2}-1,i-2}\, \xi_{i}.
 \end{split}
 \eq
%    If we note that
%\[\left\lfloor
% \tfrac{i-1}{2}\right\rfloor=\begin{cases}
%  \left\lfloor
% \frac{i}{2}\right\rfloor,&i\ \text{odd},\\
% \left\lfloor
% \frac{i}{2}\right\rfloor-1,&i\ \text{even},
% \end{cases}\]
It is not difficult to see that \eqref{ind24} can be written alternatively as
\[
\psi^{r,s}_{i}(x)=i\, \dd\sum_{m=0}^{\left\lfloor
 \frac{i}{2}\right\rfloor}\frac{s^m\, \left(\dd\frac{r}{a}\right)^{i-2 m}\, \binom{i-m}{m}}{i-m}\ _{2} F_{1}\left.\left(
\begin{array}{cc}
 -m,i-m\\
  i-2m+2
\end{array}
\right|\dd\frac{b\, r^2}{a^2\, s}\right)\ \phi^{a,b}_{i-2m}(x).
\]
 The proof of Theorem \ref{connecthm2} is now complete.\\
 \end{proof}
  \noindent
  Next, we give a special case of formula \eqref{connect1}.
  \begin{cor}
  \label{Cor1}
   For $r=a,s=b$, the connection formula \eqref{connect1} reduces to the following one
 \bq
 \label{sp1}
\phi^{a,b}_{i}(x)=\dd\sum_{m=0}^{\left\lfloor
 \frac{i}{2}\right\rfloor}c_{i-2 m}\, (-b)^m\, \psi^{a,b}_{i-2m}(x).
 \eq
  \end{cor}
  \begin{proof}
   The substitution by $r=a$ and $s=b$ into relation \eqref{connect1} yields the following relation:
   \[
 \label{connect11}
\phi^{a,b}_{i}(x)=\dd\sum_{m=0}^{\left\lfloor
 \frac{i}{2}\right\rfloor}c_{i-2 m}\, (-b)^m\, \binom{i}{m} \ _{2} F_{1}\left.\left(
\begin{array}{cc}
 -m,m-i\\
  -i
\end{array}
\right|1\right)\ \psi^{r,s}_{i-2m}(x).
\]
 With the aid of Chu-Vandermond identity, it is easy to see that
 \[\ _{2} F_{1}\left.\left(
\begin{array}{cc}
 -m,m-i\\
  -i
\end{array}
\right|1\right)=\dd\frac{1}{\binom{i}{m}},\]
and consequently, formula \eqref{sp1} can be easily obtained.
  \end{proof}
   Taking into consideration the special polynomials of the two classes $\phi^{a,b}_{j}(x), \psi^{r,s}_{j}(x)$ mentioned in Table 1, several connection formulae can be deduced as special cases of Corollary \ref{Cor1}.
  \begin{cor}
  \label{CorSp1} For a non-negative integer $i$, the following
   connection formulae hold
   \begingroup
\allowdisplaybreaks
 \begin{align}
F_{i+1}(x)=&\dd\sum_{m=0}^{\left\lfloor
 \frac{i}{2}\right\rfloor}c_{i-2 m}\, (-1)^m\, L_{i-2m}(x),
\\
P_{i+1}(x)=&\dd\sum_{m=0}^{\left\lfloor
 \frac{i}{2}\right\rfloor}c_{i-2 m}\, (-1)^m\, Q_{i-2m}(x),
 \\
\mathcal{F}_{i+1}(x)=&\dd\sum_{m=0}^{\left\lfloor
 \frac{i}{2}\right\rfloor}c_{i-2 m}\, 2^m\, f_{i-2m}(x),
\\
\label{UT}
  U_{i}(x)=&2 \sum _{m=0}^{\left\lfloor \frac{i}{2}\right\rfloor } c_{i-2 m}\ T_{i-2 m}(x),\\
E^{\alpha}_{i+1}(x)=&\dd\sum_{m=0}^{\left\lfloor
 \frac{i}{2}\right\rfloor}c_{i-2 m}\, \alpha^m\, D^{\alpha}_{i-2m}(x).
  \end{align}
  \endgroup
  \end{cor}
  \begin{rem}
   The connection formula in \eqref{UT} can be translated to the following trigonometric identity
   \bq
   2\,  \sin\theta\, \sum _{m=0}^{\left\lfloor \frac{i}{2}\right\rfloor } c_{i-2 m}\, \cos(i-2m)\, \theta=\sin(i+1)\, \theta.
   \eq
  \end{rem}
   \begin{cor}
   For $r=a,s=b$, the connection formula \eqref{connect2} reduces to the following connection formula
 \bq
 \label{sp2}
\psi^{a,b}_{i}(x)=\phi^{a,b}_{i}(x)+b\, \phi^{a,b}_{i-2}(x).
 \eq
  \end{cor}
  \begin{proof}
   The substitution by $r=1,b=s$ into relation \eqref{connect2} yields the following relation:
   \bq
 \label{connect55}
\psi^{a,b}_{i}(x)=i\, \dd\sum_{m=0}^{\left\lfloor
 \frac{i}{2}\right\rfloor}\frac{b^m\, \binom{i-m}{m}}{i-m}\ _{2} F_{1}\left.\left(
\begin{array}{cc}
 -m,i-m\\
  i-2m+2
\end{array}
\right|1\right)\ \phi^{a,b}_{i-2m}(x).
 \eq
Noting that
  \[\ _{2} F_{1}\left.\left(
\begin{array}{cc}
 -m,i-m\\
  i-2m+2
\end{array}
\right|1\right)=\begin{cases}
 1,&m=0,\\
 \frac{1}{i},&m=1,\\
 0,&\text{otherwise},
\end{cases}\]
  it is easy to see that formula \eqref{connect55} reduces to
   \bq
\psi^{a,b}_{i}(x)=\phi^{a,b}_{i}(x)+b\, \phi^{a,b}_{i-2}(x).
 \eq
  \end{proof}
 Taking into consideration the special polynomials of the two classes $\phi^{a,b}_{j}(x), \psi^{r,s}_{j}(x)$, several simple connection formulae can be deduced as special cases of relation \eqref{sp2}.
  \begin{cor}
  \label{CorSp2}
   For every non-negative integer $i$, the following connection formulae hold
 \begin{align*}
L_{i}(x)=&F_{i+1}(x)+F_{i-1}(x),&&
Q_{i}(x)=P_{i+1}(x)+P_{i-1}(x),\\
f_{i}(x)=&\mathcal{F}_{i+1}-2\, \mathcal{F}_{i-1}(x),&&
  2\, T_{i}(x)=U_{i}(x)-U_{i-2}(x),\\
D^{\alpha}_{i}(x)=&E^{\alpha}_{i}(x)-\alpha\, E^{\alpha}_{i-2}(x).
  \end{align*}
  \end{cor}

  \section{Connection formulae between two different generalized polynomials in the same class}
  This section is concerned with introducing other new connection formulae. We give connection formulae between two different generalized Fibonacci polynomials. Connection formulae between two generalized Lucas polynomials of different parameters are also given. We will show again that all the connection coefficients are expressed in terms of $_2F_{1}(z)$ for certain $z$.
  % The results are given without details in the following two theorems.
\begin{thm}
 \label{connecthm3}
  For every non-negative integer $i$, the following connection formula holds
 \bq
 \label{connect3}
\phi^{a,b}_{i}(x)=\left(\dd\frac{a}{r}\right)^i\, \dd\sum_{m=0}^{\left\lfloor
 \frac{i}{2}\right\rfloor}\frac{(-s)^{m}\, \, \binom{i}{m}\, (i-2m+1)}{i-m+1}\ _{2} F_{1}\left.\left(
\begin{array}{cc}
 -m,-i+m-1\\
  -i
\end{array}
\right|\dd\frac{b\, r^2}{a^2\, s}\right)\ \phi^{r,s}_{i-2m}(x).
 \eq
 \end{thm}
 \begin{thm}
 \label{connecthm4}
  For every non-negative integer $i$, the following connection formula holds
 \bq
 \label{connect4}
\psi^{a,b}_{i}(x)=\left(\dd\frac{a}{r}\right)^i\, \dd\sum_{m=0}^{\left\lfloor
 \frac{i}{2}\right\rfloor}c_{i-2m}\, (-s)^m\,  \binom{i}{m} \ _{2} F_{1}\left.\left(
\begin{array}{cc}
 -m,m-i\\
  1-i
\end{array}
\right|\dd\frac{b\, r^2}{a^2\, s}\right)\ \psi^{r,s}_{i-2m}(x).
 \eq
 \end{thm}
 \begin{proof}
 Theorems \ref{connecthm3} and \ref{connecthm4} can be proved using similar procedures followed in the proof of Theorem \ref{connecthm2}, but we give here another strategy of proof which is built on making use of the power form representation of each of $\phi^{a,b}_{i}(x),\psi^{a,b}_{i}(x),\ i\ge 0,$ and their inversion formulae. Now, we are going to prove Theorem \ref{connecthm3}. Making use of the analytic form of $\phi^{a,b}_{i}(x)$ in \eqref{powerphi} together with its inversion formula in \eqref{invGenFib} leads to the following formula
  \[\begin{split}
  \phi^{a,b}_{i}(x)=&\sum _{r=0}^{\left\lfloor \frac{i}{2}\right\rfloor }\ b^r\, \binom{i-r}{r} \left(\frac{a}{r}\right)^{i-2 r}\,
  \sum _{\ell=0}^{\left\lfloor \frac{i}{2}\right\rfloor -r}\ \frac{(-s)^\ell (i-2 \ell-2 r+1) \binom{i-2 r}{\ell}}{i-\ell-2 r+1}\ \phi^{r,s}_{i-2\ell-2\, r}(x).\end{split}\]
  Rearranging the right-hand side of the last relation, and performing some lengthy manipulations enables one to obtain the following relation
     \bq
     \label{conncstep}
     \begin{split}
 \phi^{a,b}_{i}(x)=&\sum _{m=0}^{\left\lfloor \frac{i}{2}\right\rfloor } (2m-i-1)\, \sum _{r=0}^{m}\frac{(-1)^{m+r}\, b^r\, \left(\frac{a}{r}\right)^{i-2 r} \, s^{m-r} \, \binom{i-r}{r}\, \binom{i-2 r}{m-r}}{m-i+r-1}\ \phi^{r,s}_{i-2m}(x).\end{split}
 \eq
   It is not difficult to see that
   \[\begin{split}
   &\sum _{r=0}^{m}\frac{(-1)^{m+r}\, b^r\, \left(\frac{a}{r}\right)^{i-2 r} \, s^{m-r} \, \binom{i-r}{r}\, \binom{i-2 r}{m-r}}{m-i+r-1}\\
   &=\frac{(-1)^{m+1}\, s^m\,  \left(\frac{a}{r}\right)^i\, \binom{i}{m}}{i-m+1}
     \ _{2} F_{1}\left.\left(
\begin{array}{cc}
 -m,-i+m-1\\
  -i
\end{array}
\right|\dd\frac{b\, r^2}{a^2\, s}\right),\end{split}\]
and consequently, relation \eqref{conncstep} yields the connection formula
\bq
\phi^{a,b}_{i}(x)=\left(\dd\frac{a}{r}\right)^i\, \dd\sum_{m=0}^{\left\lfloor
 \frac{i}{2}\right\rfloor}\frac{(-s)^{m}\, \, \binom{i}{m}\, (i-2m+1)}{i-m+1}\ _{2} F_{1}\left.\left(
\begin{array}{cc}
 -m,-i+m-1\\
  -i
\end{array}
\right|\dd\frac{b\, r^2}{a^2\, s}\right)\ \phi^{r,s}_{i-2m}(x).
 \eq
 \end{proof}
\begin{rem}
 Several connection formulae can be deduced as special cases of the two connection formulae in \eqref{connect3} and \eqref{connect4}. In fact, there are forty relations that can be obtained.
% Twenty of them are special cases of \eqref{connect3} and the other twenty are special cases of \eqref{connect4}.
 In the following, we give some of these formulae.

 \end{rem}
 \begin{cor}
   If we set $r=2,s=-1$ in relation \eqref{connect3}, then we obtain
 \bq
 \label{connect3Sp1}
\phi^{a,b}_{i}(x)=\left(\dd\frac{a}{2}\right)^i\, \dd\sum_{m=0}^{\left\lfloor
 \frac{i}{2}\right\rfloor}\frac{\binom{i}{m}\, (i-2m+1)}{i-m+1}\ _{2} F_{1}\left.\left(
\begin{array}{cc}
 -m,-i+m-1\\
  -i
\end{array}
\right|\dd\frac{4\, b}{-a^2}\right)\ U_{i-2m}(x).
 \eq
 \end{cor}
 \begin{cor}
 \label{Coro5}
   If we set $a=2,b=-1$ in relation \eqref{connect3}, then we obtain
  \bq
 \label{connect3Sp2}
U_{i}(x)=\left(\dd\frac{2}{r}\right)^i\, \dd\sum_{m=0}^{\left\lfloor
 \frac{i}{2}\right\rfloor}\frac{(-s)^{m}\, \, \binom{i}{m}\, (i-2m+1)}{i-m+1}\ _{2} F_{1}\left.\left(
\begin{array}{cc}
 -m,-i+m-1\\
  -i
\end{array}
\right|\dd\frac{-r^2}{4\, s}\right)\ \phi^{r,s}_{i-2m}(x).
 \eq
 In particular, and if we set $a=b=1$ in \eqref{connect3Sp1}, and $r=s=1$ in \eqref{connect3Sp2}, then the following two relations are obtained:
 \bq
 \label{connect3Sp3}
F_{i+1}(x)=\left(\dd\frac{1}{2}\right)^i\, \dd\sum_{m=0}^{\left\lfloor
 \frac{i}{2}\right\rfloor}\frac{\binom{i}{m}\, (i-2m+1)}{i-m+1}\ _{2} F_{1}\left.\left(
\begin{array}{cc}
 -m,-i+m-1\\
  -i
\end{array}
\right|-4\right)\ U_{i-2m}(x),
 \eq
 and
  \bq
 \label{connect3Sp4}
U_{i}(x)=2^i\, \dd\sum_{m=0}^{\left\lfloor
 \frac{i}{2}\right\rfloor}\frac{(-1)^{m}\, \, \binom{i}{m}\, (i-2m+1)}{i-m+1}\ _{2} F_{1}\left.\left(
\begin{array}{cc}
 -m,-i+m-1\\
  -i
\end{array}
\right|\dd\frac{-1}{4}\right)\ F_{i-2m+1}(x).
 \eq
 \end{cor}
 \begin{rem}
  The two relations in \eqref{connect3Sp3} and \eqref{connect3Sp4} are in agreement with those recently developed in \cite{RamFib}.
 \end{rem}
 \begin{cor}
  If we set $r=2,s=-1$  in \eqref{connect4}, then we obtain
 \bq
 \label{connect4Sp11}
\psi^{a,b}_{i}(x)=2\, \left(\dd\frac{a}{2}\right)^i\, \dd\sum_{m=0}^{\left\lfloor
 \frac{i}{2}\right\rfloor}c_{i-2m}\,  \binom{i}{m} \ _{2} F_{1}\left.\left(
\begin{array}{cc}
 -m,m-i\\
  1-i
\end{array}
\right|\dd\frac{-4\, b}{a^2}\right)\ T_{i-2m}(x).
 \eq
 \end{cor}
% \begin{cor}
% \label{Coro7}
%  If we set $a=2,b=-1$  in \eqref{connect4}, then one obtains
% \bq
% \label{connect4Sp22}
%T_{i}(x)=\frac12\, \left(\dd\frac{2}{r}\right)^i\, \dd\sum_{m=0}^{\left\lfloor
% \frac{i}{2}\right\rfloor}c_{i-2m}\, (-s)^m\,  \binom{i}{m} \ _{2} F_{1}\left.\left(
%\begin{array}{cc}
% -m,m-i\\
%  1-i
%\end{array}
%\right|\dd\frac{-r^2}{4\, s}\right)\ \psi^{r,s}_{i-2m}(x).
% \eq
% In particular, and if we set $a=b=1$ in \eqref{connect4Sp11} and $r=s=1$ in \eqref{connect4Sp22}, respectively, then the following two relations are obtained:
%  \bq
% \label{connect4Sp1}
%L_{i}(x)=\frac{1}{2^{i-1}}\, \dd\sum_{m=0}^{\left\lfloor
% \frac{i}{2}\right\rfloor}c_{i-2m}\, \binom{i}{m} \ _{2} F_{1}\left.\left(
%\begin{array}{cc}
% -m,m-i\\
%  1-i
%\end{array}
%\right|-4\right)\ T_{i-2m}(x),
% \eq
% and
% \bq
% \label{connect4Sp2}
%T_{i}(x)=2^{i-1}\, \dd\sum_{m=0}^{\left\lfloor
% \frac{i}{2}\right\rfloor}c_{i-2m}\, (-1)^m\,  \binom{i}{m} \ _{2} F_{1}\left.\left(
%\begin{array}{cc}
% -m,m-i\\
%  1-i
%\end{array}
%\right|\dd\frac{-1}{4}\right)\ L_{i-2m}(x).
% \eq
% \end{cor}
 % \begin{rem}
%  Two relations \eqref{connect4Sp1} and \eqref{connect4Sp2} are
%  in agreement with those recently developed in \cite{RamFib}.
% \end{rem}
% \begin{rem}
% Some of the results of Sections 3 and 4 generalize the results obtained in Ref. \cite{RamFib}, and the other results are new.
% \end{rem}
\section{Two applications of the introduced connection formulae}
  In this section, we present two applications of the connection formulae derived in Sections 3 and 4. In the first, we give some new relations between some celebrated numbers and in the second application, some new definite weighted integrals are given.
  \subsection{Formulae between some celebrated numbers}
  Theorems \ref{connecthm1}-\ref{connecthm4} enable one to join two families of numbers whether they either belong to the same class of generalized Fibonacci or Lucas numbers or they belong to two different such classes. With respect to the six families of numbers $\{F_{i}\}_{i\ge 0},\{P_{i}\}_{i\ge 0},\{\mathcal{F}_{i}\}_{i\ge 0},\{L_{i}\}_{i\ge 0},\{Q_{i}\}_{i\ge 0},\{f_{i}\}_{i\ge 0}$, defined in Table 1, there are thirty relations link between them. All the coefficients between them involve $_2F_{1}(z)$ for certain $z$. In six cases only, the appearing $_2F_{1}(z)$ can be summed. They are given explicitly in the following Corollary.
  %These relations can be followed easily
%  from Corollaries \ref{CorSp1} and \ref{CorSp2} only by setting $x=1$. The following corollary exhibits these formulae.
  \begin{cor}
  For every non-negative integer $i$, the following identities are valid:
%  in the same class or in two different class with respect to the six First, some famous numbers can be expressed in terms of another set of numbers. For example, and if we set $x=1$, in identities \eqref{}-\eqref{}, the following expressions are obtained
   \begin{align*}
F_{i+1}=&\dd\sum_{m=0}^{\left\lfloor
 \frac{i}{2}\right\rfloor}c_{i-2 m}\, (-1)^m\, L_{i-2m},&&
P_{i+1}=\dd\sum_{m=0}^{\left\lfloor
 \frac{i}{2}\right\rfloor}c_{i-2 m}\, (-1)^m\, Q_{i-2m},\\
\mathcal{F}_{i+1}=&\dd\sum_{m=0}^{\left\lfloor
 \frac{i}{2}\right\rfloor}c_{i-2 m}\, 2^m\, f_{i-2m},&& L_{i}=F_{i+1}+F_{i-1},\\
Q_{i}=&P_{i+1}+P_{i-1},&&f_{i}=\mathcal{F}_{i+1}-2\, \mathcal{F}_{i-1}.
\end{align*}
\end{cor}
 Now, with respect to the remaining twenty-four relations among the celebrated sequences of numbers, they can be deduced easily by setting $x=1$ in the relations of Theorems \ref{connecthm1}-\ref{connecthm4} taking into consideration the different special cases of the two generalized polynomials $\{\phi^{a,b}_{i}(x)\}_{i\ge 0}$ and $\{\psi^{a,b}_{i}(x)\}_{i\ge 0}$.
 %In the following corollary, we write two of these formulae.
% \begin{cor}
% For every non-negative integer $i$, the following identities are valid:
%\[
%F_{i+1}=\left(\dd\frac{1}{2}\right)^i\, \dd\sum_{m=0}^{\left\lfloor
% \frac{i}{2}\right\rfloor}\frac{(-1)^{m}\, \, \binom{i}{m}\, (i-2m+1)}{i-m+1}\ _{2} F_{1}\left.\left(
%\begin{array}{cc}
% -m,-i+m-1\\
%  -i
%\end{array}
%\right|4\right)\ Q_{i-2m},
%\]
%and
%\[
%L_{i}=i\, \dd\sum_{m=0}^{\left\lfloor
% \frac{i}{2}\right\rfloor}\frac{\left(\dd\tfrac{1}{3}\right)^{i-2 m}\, \binom{i-m}{m}}{i-m}\ _{2} F_{1}\left.\left(
%\begin{array}{cc}
% -m,i-m\\
%  i-2m+2
%\end{array}
%\right|\dd\frac{-2}{9}\right)\ {\mathcal{F}}_{i-2m}.
% \]
% \end{cor}
 \subsection{Some definite weighted integrals}
In this section, and based on the connection formulae developed in Sections 3 and 4, some definite weighted integrals are given in terms of certain hypergeometric functions of the type $_2F_{1}(z)$. Some of these results are given in the following two corollaries.
 \begin{cor}
     For all non-negative integers $i,j$ and $j\ge i$, the following two integrals
   formulae hold:
   \bq
  \begin{split}
  \label{IntegraL1}
 \displaystyle\int\limits_{-1}^{1}\dd\frac{1}{\sqrt{1-x^2}}\ \phi^{a,b}_{i}(x)\  T_j(x)\, dx=
  \begin{cases}
  \pi\, \left(\dd\frac{a}{2}\right)^i \, (-1)^{i+j}\, \dd\binom{i}{\frac{i-j}{2}} \,  _{2}F_{1}\left.\left(\begin{array}{cc}
 -\left(\frac{i+j}{2}\right)\,\frac{j-i}{2}\\ -i
\end{array}
\right|\frac{-4\, b}{a^2}\right),& \textrm{if }(i+j)\ \text{even},\\
 0,& \text{otherwise},
 \end{cases}
 \end{split}
 \eq
 % \bq
%  \begin{split}
%  \label{IntegraL1}
%  &\hspace*{60pt}\displaystyle\int\limits_{-1}^{1}\dd\frac{1}{\sqrt{1-x^2}}\ \phi^{a,b}_{i}(x)\  T_j(x)\, dx=\\
%  &\begin{cases}
%  \pi\, \left(\dd\frac{a}{2}\right)^i \, (-1)^{i+j}\, \dd\binom{i}{\frac{i-j}{2}} \,  _{2}F_{1}\left.\left(\begin{array}{cc}
% -\left(\frac{i+j}{2}\right)\,\frac{j-i}{2}\\ -i
%\end{array}
%\right|\frac{-4\, b}{a^2}\right),& \textrm{if }(i+j)\ \text{even},\\
% 0,& \text{otherwise},
% \end{cases}
% \end{split}
% \eq
 and
 \bq
  \begin{split}
  \label{IntegraL2}
\displaystyle\int\limits_{-1}^{1}\sqrt{1-x^2}\ \psi^{r,s}_{i}(x)\  U_j(x)\, dx=
    \begin{cases}
  \frac{\pi\,  i\, \left(\dd\frac{r}{2}\right)^j \dd\binom{\frac{i+j}{2}}{\frac{i-j}{2}} s^{\frac{i-j}{2}}}{i+j}\, \ _{2}F_{1}\left.\left(\begin{array}{cc}
 \frac{j-i}{2},\frac{i+j}{2}\\ j+2
\end{array}
\right|\frac{-r^2}{4\, s}\right),& \textrm{if }(i+j)\ \text{even},\\
 0,& \text{otherwise}.
 \end{cases}
 \end{split}
 \eq
  %\bq
%  \begin{split}
%  \label{IntegraL2}
%  &\hspace*{60pt}\displaystyle\int\limits_{-1}^{1}\sqrt{1-x^2}\ \psi^{r,s}_{i}(x)\  U_j(x)\, dx=\\
%   & \begin{cases}
%  \frac{\pi\,  i\, \left(\dd\frac{r}{2}\right)^j \dd\binom{\frac{i+j}{2}}{\frac{i-j}{2}} s^{\frac{i-j}{2}}}{i+j}\, \ _{2}F_{1}\left.\left(\begin{array}{cc}
% \frac{j-i}{2},\frac{i+j}{2}\\ j+2
%\end{array}
%\right|\frac{-r^2}{4\, s}\right),& \textrm{if }(i+j)\ \text{even},\\
% 0,& \text{otherwise}.
% \end{cases}
% \end{split}
% \eq
  \end{cor}
  \begin{proof}
  If we set $r=2,s=-1$, in relation \eqref{connect1}, then after making use of the orthogonality relation of $T_{j}(x)$, relation \eqref{IntegraL1} can be obtained. Relation \eqref{IntegraL2} can be obtained by setting $a=2,b=-1$  in relation \eqref{connect2} and making use of the orthogonality relation of $U_{j}(x)$.
  \end{proof}
  %\begin{cor}
%  \label{Coro12}
%       For all non-negative integers $i,j$ and $j\ge i$, the following two integrals
%   formulae hold:
%\bq
%  \begin{split}
%  \label{I3}
%  &\hspace*{60pt} \dd\int\limits_{-1}^{1}\sqrt{1-x^2}\ \phi^{a,b}_{i}(x)\  U_j(x)\, dx=\\
%      &\begin{cases}
%      \frac{\pi  \left(\dd\frac{a}{2}\right)^i (j+1) (-1)^{i+j} \dd\binom{i}{\frac{i-j}{2}}}{i+j+2}\,
%_{2}F_{1}\left.\left(\begin{array}{cc}
% -\left(\frac{i+j+2}{2}\right),\frac{j-i}{2}\\ -i
%\end{array}
%\right|\frac{-4\, b}{a^2}\right),& \textrm{if }(i+j)\ \text{even},\\
% 0,& \text{otherwise}.
% \end{cases}
%\end{split}
%\eq
% \bq
%  \begin{split}
%  \label{I4}
%  &\hspace*{60pt} \dd\int\limits_{-1}^{1}\dd\frac{1}{\sqrt{1-x^2}}\ \psi^{a,b}_{i}(x)\  T_j(x)\, dx=\\
%      &\begin{cases}
%  \pi  \left(\frac{a}{2}\right)^i\, (-1)^{i+j} \binom{i}{\frac{i-j}{2}}\,  _{2}F_{1}\left.\left(\begin{array}{cc}
% -\left(\frac{i+j}{2}\right)\,\frac{j-i}{2}\\ 1-i
%\end{array}
%\right|\frac{-4\, b}{a^2}\right),& \textrm{if }(i+j)\ \text{even},\\
% 0,& \text{otherwise}.
% \end{cases}
%  \end{split}
%\eq
%\end{cor}
% \begin{proof}
%  The two formulae of Corollary \ref{Coro12} are immediate consequences of Corollaries \ref{Coro5} and \ref{Coro7}.
% \end{proof}
   \section{Reduction of some odd and even radicals}
   In this section, we introduce some new reduction formulae of certain odd and even radicals through the employment of the two classes of generalized Fibonacci and Lucas polynomials introduced in Section 2. The problems of the reduction of radicals are of interest. There are considerable contributions regarding the reduction of various kinds of radicals (see for example,
\cite{witula2010cardano,dattoli2015cardan,landau1992simplification,
osler2001cardan,berndt1997ramanujan,berndt1998radicals}).  First, the following lemma is basic in the sequel.
\begin{lem}
   \label{lem2}
   For every non-negative integer $k$, the following formula is valid:
   \bq
   \label{squarere}
    (a^2\, x^2+4\, b)\left(\phi^{a,b}_{k-1}(x)\right)^2-\left(\psi^{a,b}_{k}(x)\right)^2=4\, (-1)^{k+1}\, b^k.
   \eq
   \end{lem}
   \begin{proof}
       The two Binet's formulae \eqref{BinetFib} and \eqref{BinetLuc} can be rewritten in the form
 \bq
 \label{phiasfrac}
 \phi^{a,b}_{k-1}(x)=\dd\frac{\left(\alpha(x)\right)^k-\left(\beta(x)\right)^k}{\alpha(x)-\beta(x)},
 \eq
 and
  \bq
   \label{psiasfrac}
 \psi^{a,b}_{k}(x)=\left(\alpha(x)\right)^k+\left(\beta(x)\right)^k,
 \eq
 where
 \bq
 \label{alphabeta}
  \alpha(x)=\frac{a\, x+\sqrt{a^2\, x^{2}+4\, b}}{2},\qquad \beta(x)=\frac{a\, x -\sqrt{a^2\, x^{2}+4\, b}}{2},
  \eq
 and consequently, we have
        \[\begin{split}
    (a^2\, x^2+4\, b)\left(\phi^{a,b}_{k-1}(x)\right)^2-\left(\psi^{a,b}_{k}(x)\right)^2=
   \left\{(\alpha(x))^k-(\beta(x))^k\right\}^2\\
     -\left\{(\alpha(x))^k+(\beta(x))^k\right\}^2
 =-4\left(\alpha(x)\, \beta(x)\right)^k=4\, (-1)^{k+1}\, b^k.
 \end{split}\]
  Lemma \ref{lem2} is now proved.
    \end{proof}
%In this respect, some new results concerning the reduction of certain radicals are established.
\subsection{New formulae of some odd radicals}
We state and prove the following theorem in which we show how to reduce some odd radicals using the two generalized Fibonacci and Lucas polynomials.
   \begin{thm}
   \label{thmoddradicaL}
Let $k$ be any positive odd integer. Then for every $x\in
\mathbb{R}$ such that $(ax)^{2}\geq-4b$, the following two identities
hold:
    \bq
    \label{oddre}
    \sqrt[k]{\tfrac{\psi^{a,b}_{k}(x)}{2}+\sqrt{\tfrac{\left(\psi^{a,b}_{k}(x)\right)^2}{4}+b^k}}
    +\sqrt[k]{\tfrac{\psi^{a,b}_{k}(x)}{2}-\sqrt{\tfrac{\left(\psi^{a,b}_{k}(x)\right)^2}{4}+b^k}}=a\, x,
    \eq
    and
        \bq
    \label{oddre1}
    \sqrt[k]{\tfrac{\psi^{a,b}_{k}(x)}{2}+\sqrt{\tfrac{\left(\psi^{a,b}_{k}(x)\right)^2}{4}+b^k}}
    -\sqrt[k]{\tfrac{\psi^{a,b}_{k}(x)}{2}-\sqrt{\tfrac{\left(\psi^{a,b}_{k}(x)\right)^2}{4}+b^k}}=\sqrt{a^2\, x^2+4\, b}.
    \eq
\end{thm}
\begin{proof}
  The idea of the proof is built on making use of the two Binet's formulae \eqref{phiasfrac} and \eqref{psiasfrac}. From \eqref{alphabeta}, it is clear that
  \bq
 \label{idsum}
\alpha(x)+\beta(x)=a\, x,
\eq
\bq
 \label{idminum}
 \alpha(x)-\beta(x)=\sqrt{a^2\, x^2+4\, b}.
 \eq
  Inserting identity \eqref{idminum} into relation \eqref{phiasfrac}, the following relation is obtained
    \bq
    \sqrt{a^2\, x^2+4\, b}\  \phi^{a,b}_{k-1}(x)=\left(\alpha(x)\right)^k-\left(\beta(x)\right)^k.
     \eq

 Now, the last relation along with relation \eqref{psiasfrac} yield the following two relations:
   \bq
   \label{eq1}
   \left(\alpha(x)\right)^{k}=\frac12\left\{\psi^{a,b}_{k}(x)+\sqrt{a^2\, x^2+4\, b}\, \phi^{a,b}_{k-1}(x)\right\},
   \eq
   and
    \bq
     \label{eq2}
   \left(\beta(x)\right)^{k}=\frac12\left\{\psi^{a,b}_{k}(x)-\sqrt{a^2\, x^2+4\, b}\, \phi^{a,b}_{k-1}(x)\right\}.
   \eq
    Next, and based on identity \eqref{squarere}, relations \eqref{eq1} and \eqref{eq2} can be written alternatively as
\[\left(\alpha(x)\right)^{k}=\frac12\left\{\psi^{a,b}_{k}(x)+\sqrt{\left(\psi^{a,b}_{k}(x)\right)^2+4\, b^k}\right\},\]
   and
\[ \left(\beta(x)\right)^{k}=\frac12\left\{\psi^{a,b}_{k}(x)-\sqrt{\left(\psi^{a,b}_{k}(x)\right)^2+4\, b^k}\right\}.\]
   Finally, we observe that the last two relations lead to the following interesting identity
\[\sqrt[k]{\tfrac{\psi^{a,b}_{k}(x)}{2}+\sqrt{\tfrac{\left(\psi^{a,b}_{k}(x)\right)^2}{4}+b^k}}
    \pm\sqrt[k]{\tfrac{\psi^{a,b}_{k}(x)}{2}-\sqrt{\tfrac{\left(\psi^{a,b}_{k}(x)\right)^2}{4}+b^k}}=\alpha(x)\pm \beta(x),\]
     which can be split using relations \eqref{idsum} and \eqref{idminum} to give the two identities \eqref{oddre} and \eqref{oddre1}.
    This proves Theorem \ref{thmoddradicaL}.
    \end{proof}
     \begin{cor}
   \label{thmodd}
Let $k$ be any positive odd integer. Then for nonzero real numbers
$a$ and $b$ with $b\geq {-1}/{4}$, there are infinite number of
expressions of unit radicals in the sense that
    \bq
    \label{oddreunit}
    \sqrt[k]{\tfrac{\psi^{a,b}_{k}(\frac{1}{a})}{2}+\sqrt{\tfrac{\left(\psi^{a,b}_{k}(\frac{1}{a})\right)^2}{4}+b^k}}
    +\sqrt[k]{\tfrac{\psi^{a,b}_{k}(\frac{1}{a})}{2}-\sqrt{\tfrac{\left(\psi^{a,b}_{k}(\frac{1}{a})\right)^2}{4}+b^k}}=1.
    \eq
\end{cor}
\begin{proof}
 Setting $x=\frac{1}{a}$ in \eqref{oddre} yields formula \eqref{oddreunit}.
\end{proof}
As a consequence of Theorem \ref{thmoddradicaL}, and taking into
consideration, the special polynomials of $\psi^{a,b}_{i}(x)$ in
Table 1, we obtain some new reduction formulae of some odd radicals
involving Lucas, Pell-Lucas, Fermat Lucas, and Chebyshev and Dickson
polynomials of the first kinds. The following corollaries display
these formulae.
    \begin{cor}
    \label{cor14}
    For every positive integer $i$, the following reduction formulae hold
    \begin{enumerate}

\item
 For all $x\in \mathbb{R}$
 \\ $\sqrt[2i+1]{\frac{
L_{2i+1}\left( x\right) }{2}+\sqrt{\frac{L_{2i+1}^{2}\left( x\right)
}{4}+1}}+\sqrt [2i+1]{\frac{L_{2i+1}\left( x\right)
}{2}-\sqrt{\frac{L_{2i+1}^{2}\left( x\right) }{4} +1}}=x,$\\
 \\ $\sqrt[2i+1]{\frac{
L_{2i+1}\left( x\right) }{2}+\sqrt{\frac{L_{2i+1}^{2}\left( x\right)
}{4}+1}}-\sqrt [2i+1]{\frac{L_{2i+1}\left( x\right)
}{2}-\sqrt{\frac{L_{2i+1}^{2}\left( x\right) }{4} +1}}=\sqrt{x^2+4}.$

\item For all $x\in \mathbb{R}$
 \\$\sqrt[2i+1]{\frac{ Q_{2i+1}\left( x\right)
}{2}+\sqrt{\frac{Q_{2i+1}^{2}\left( x\right) }{4}+1}}+\sqrt
[2i+1]{\frac{Q_{2i+1}\left( x\right)
}{2}-\sqrt{\frac{Q_{2i+1}^{2}\left( x\right) }{4} +1}}=2\, x ,$\\
\\$\sqrt[2i+1]{\frac{ Q_{2i+1}\left( x\right)
}{2}+\sqrt{\frac{Q_{2i+1}^{2}\left( x\right) }{4}+1}}-\sqrt
[2i+1]{\frac{Q_{2i+1}\left( x\right)
}{2}-\sqrt{\frac{Q_{2i+1}^{2}\left( x\right) }{4} +1}}=2\, \sqrt{x^2+1}.$

\item For all $x\in {\mathbb{R}\setminus
(-\frac{2\sqrt2}{3},\frac{2\sqrt2}{3})}$ \\ $\sqrt[2i+1]{\frac{
f_{2i+1}\left( x\right) }{2}+\sqrt{\frac{f_{2i+1}^{2}\left( x\right)
}{4}-2^{2i+1}}}+\sqrt [2i+1]{\frac{f_{2i+1}\left( x\right)
}{2}-\sqrt{\frac{f_{2i+1}^{2}\left( x\right) }{4} -2^{2i+1}}}=3\,
x,$\\
$\sqrt[2i+1]{\frac{
f_{2i+1}\left( x\right) }{2}+\sqrt{\frac{f_{2i+1}^{2}\left( x\right)
}{4}-2^{2i+1}}}-\sqrt [2i+1]{\frac{f_{2i+1}\left( x\right)
}{2}-\sqrt{\frac{f_{2i+1}^{2}\left( x\right) }{4} -2^{2i+1}}}=\sqrt{9\, x^2-8}.$
\item For all $x\in
{\mathbb{R}\setminus (-1,1)}$\\ $\sqrt[2i+1]{{ T_{2i+1}\left(
x\right) }+\sqrt{{T_{2i+1}^{2}\left( x\right) }-1}}+\sqrt[2i+1]{{
T_{2i+1}\left( x\right) }-\sqrt{{T_{2i+1}^{2}\left( x\right)
}-1}}=2\, x, $\\
$\sqrt[2i+1]{{ T_{2i+1}\left(
x\right) }+\sqrt{{T_{2i+1}^{2}\left( x\right) }-1}}-\sqrt[2i+1]{{
T_{2i+1}\left( x\right) }-\sqrt{{T_{2i+1}^{2}\left( x\right)
}-1}}=2\, \sqrt{x^2-1}.$
\item Either for all $x\in \mathbb{R}$ whenever $\alpha<0$ or for all $x\in {\mathbb{R}\setminus
(-2\, \sqrt{\alpha},2\, \sqrt {\alpha})}$ whenever $\alpha>0$ \\
$\sqrt[2i+1]{\frac{ D_{2i+1}^{\alpha}\left( x\right)
}{2}+\sqrt{\frac{({D_{2i+1}^{\alpha}\left( x\right)
})^{2}}{4}-\alpha^{2i+1}}}+\sqrt
[2i+1]{\frac{D_{2i+1}^{\alpha}\left( x\right)
}{2}-\sqrt{\frac{({D_{2i+1}^{\alpha}\left( x\right) })^{2}}{4}
-\alpha^{2i+1}}}= x,$\\
$\sqrt[2i+1]{\frac{ D_{2i+1}^{\alpha}\left( x\right)
}{2}+\sqrt{\frac{({D_{2i+1}^{\alpha}\left( x\right)
})^{2}}{4}-\alpha^{2i+1}}}-\sqrt
[2i+1]{\frac{D_{2i+1}^{\alpha}\left( x\right)
}{2}-\sqrt{\frac{({D_{2i+1}^{\alpha}\left( x\right) })^{2}}{4}
-\alpha^{2i+1}}}=\sqrt{x^2-4\, \alpha}.$
\end{enumerate}
\end{cor}
If we put $x=1$ in the first three items of the above corollary, we get the following
 new interesting reduction formulae of some odd radicals involving Lucas, Pell-Lucas, and Fermat-Lucas numbers.
 \begin{cor}
For every non-negative integer $i$, one has
\begin{align*}
&\sqrt[2i+1]{\frac{ L_{2i+1} }{2}+\sqrt{\frac{L_{2i+1}^{2}
}{4}+1}}+\sqrt [2i+1]{\frac{L_{2i+1} }{2}-\sqrt{\frac{L_{2i+1}^{2}
}{4} +1}}=1,\\
&\sqrt[2i+1]{\frac{ L_{2i+1} }{2}+\sqrt{\frac{L_{2i+1}^{2}
}{4}+1}}-\sqrt [2i+1]{\frac{L_{2i+1} }{2}-\sqrt{\frac{L_{2i+1}^{2}
}{4} +1}}=\sqrt{5},\\
 &\sqrt[2i+1]{\frac{ Q_{2i+1} }{2}+\sqrt{\frac{Q_{2i+1}^{2}
}{4}+1}}+\sqrt [2i+1]{\frac{Q_{2i+1} }{2}-\sqrt{\frac{Q_{2i+1}^{2}
}{4} +1}}=2,\\
&\sqrt[2i+1]{\frac{ Q_{2i+1} }{2}+\sqrt{\frac{Q_{2i+1}^{2}
}{4}+1}}-\sqrt [2i+1]{\frac{Q_{2i+1} }{2}-\sqrt{\frac{Q_{2i+1}^{2}
}{4} +1}}=2\, \sqrt{2},\\
 &\sqrt[2i+1]{\frac{ f_{2i+1} }{2}+\sqrt{\frac{f_{2i+1}^{2}
}{4}-2^{2i+1}}}+\sqrt [2i+1]{\frac{f_{2i+1}
}{2}-\sqrt{\frac{f_{2i+1}^{2} }{4} -2^{2i+1}}}=3,\\
&\sqrt[2i+1]{\frac{ f_{2i+1} }{2}+\sqrt{\frac{f_{2i+1}^{2}
}{4}-2^{2i+1}}}-\sqrt [2i+1]{\frac{f_{2i+1}
}{2}-\sqrt{\frac{f_{2i+1}^{2} }{4} -2^{2i+1}}}=1.
\end{align*}
 \end{cor}
  In the following example, we give reduction formulae of a few radicals as an application of Corollary \ref{cor14}.
  \begin{exa}
 The following identities are valid:
\begin{enumerate}
\item $\sqrt[3]{7+5\sqrt{2}}+\sqrt[3]{7-5\sqrt{2}}=
\\\\
\sqrt[3]{\frac{L_{3}(2)}{2}+\sqrt{\frac{(L_{3}(2))^{2}
}{4}+1}}+\sqrt [3]{\frac{L_{3}(2)}{2}-\sqrt{\frac{(L_{3}(2))^{2}
}{4}+1}}=2.$
\item $\sqrt[3]{7+5\sqrt{2}}-\sqrt[3]{7-5\sqrt{2}}=
\\\\
\sqrt[3]{\frac{L_{3}(2)}{2}+\sqrt{\frac{(L_{3}(2))^{2}
}{4}+1}}-\sqrt [3]{\frac{L_{3}(2)}{2}-\sqrt{\frac{(L_{3}(2))^{2}
}{4}+1}}=2\, \sqrt{2}.$
\item $\sqrt[3]{117+37\sqrt{10}}+\sqrt[3]{117-37\sqrt{10}}=
\\\\
\sqrt[3]{\frac{Q_{3}(3)}{2}+\sqrt{\frac{(Q_{3}(3))^{2}
}{4}+1}}+\sqrt [3]{\frac{Q_{3}(3)}{2}-\sqrt{\frac{(Q_{3}(3))^{2}
}{4}+1}}=6.$
\item
$\sqrt[5]{115896+19876\sqrt{34}}+\sqrt[5]{115896-19876\sqrt{34}}=
\\\\
\sqrt[5]{\frac{f_{5}(4)}{2}+\sqrt{\frac{(f_{5}(4))^{2}
}{4}-2^5}}+\sqrt [5]{\frac{f_{5}(4)}{2}-\sqrt{\frac{(f_{5}(4))^{2}
}{4}-2^5}}=12.$
\item
$\sqrt[5]{47525+19402 \sqrt{6}}+\sqrt[5]{47525-19402 \sqrt{6}}=
\\\\
\sqrt[5]{T_{5}(5)+\sqrt{(T_{5}(5))^{2}-1}}+\sqrt
[5]{T_{5}(5)-\sqrt{(T_{5}(5))^{2}-1}}=10.$

\end{enumerate}
\end{exa}
\subsection{New reduction formulae of some even radicals}
This section is devoted to presenting new formulae of some even radicals.
  \begin{thm}
   \label{thmevenradical1}
Let $k$ be any positive even integer. Then for every real number
with $(ax)^2\geq-4b$, the following two identities hold:
    \bq
    \label{Evenre}
    \begin{split}
    \sqrt[k]{\tfrac{\psi^{a,b}_{k}(x)}{2}+\sqrt{\tfrac{\left(\psi^{a,b}_{k}(x)\right)^2}{4}-b^k}}
    -\sqrt[k]{\tfrac{\psi^{a,b}_{k}(x)}{2}-\sqrt{\tfrac{\left(\psi^{a,b}_{k}(x)\right)^2}{4}-b^k}}=\begin{cases}
     \left|a\, x\right|,&b>0,\\[0.2cm]
     \sqrt{a^2\, x^2+4\, b},&b<0,
    \end{cases}
    \end{split}
    \eq
    and
        \bq
    \label{Evenre1}
    \begin{split}
    \sqrt[k]{\tfrac{\psi^{a,b}_{k}(x)}{2}+\sqrt{\tfrac{\left(\psi^{a,b}_{k}(x)\right)^2}{4}-b^k}}
    +\sqrt[k]{\tfrac{\psi^{a,b}_{k}(x)}{2}-\sqrt{\tfrac{\left(\psi^{a,b}_{k}(x)\right)^2}{4}-b^k}}=\begin{cases}
     \sqrt{a^2\, x^2+4\, b},&b>0,\\[0.2cm]
    \left|a\, x\right| ,&b<0.
    \end{cases}
    \end{split}
    \eq
\end{thm}
\begin{proof}
   Let $k$ be a positive even integer. In view of equation \eqref{BinetFib}, for all real number $x$ with
   $(ax)^2\geq-4b$:\quad
   $\phi^{a,b}_{k-1}(x)\geq 0$ if and only if $a\, x\geq 0$.
   \\Now equation \eqref{squarere} may be written as
   ${\sqrt{a^{2}x^{2}+4b}}\mid\phi_{k-1}\mid=\sqrt{\psi_{k}^{2}-4\, b^k}$.
   Using similar procedures followed in the proof of Theorem
\ref{thmoddradicaL} leads to the following identity
  % \[\sqrt{a^2\, x^2+4\, b}\  \phi^{a,b}_{k}(x)=\sqrt{\left(\psi^{a,b}_{k}(x)\right)^2+4\, b^k}.\]
\[
    \sqrt[k]{\tfrac{\psi^{a,b}_{k}(x)}{2}+\sqrt{\dd\tfrac{\left(\psi^{a,b}_{k}(x)\right)^2}{4}-b^k}}
    \mp \sqrt[k]{\tfrac{\psi^{a,b}_{k}(x)}{2}-
    \sqrt{\dd\tfrac{\left(\psi^{a,b}_{k}(x)\right)^2}{4}-b^k}}=\big|\left|\alpha(x)\right|
    \mp\left|\beta(x)\right|\big|,
\]
 and consequently, identities \eqref{Evenre} and \eqref{Evenre1} can be obtained.
    \end{proof}

As consequences of Theorem \ref{thmevenradical1}, we give some
reduction formulae of certain even radicals. The following corollary
displays these formulae.
 \begin{cor}
  \label{evenrad1}
    For every positive integer $i$ and for every real number $x$, the following formulae hold
\begin{enumerate}
\item For all $x\in \mathbb{R}$\\
$\sqrt[2i]{\frac{L_{2i}\left(
x\right)}{2}+\sqrt{\frac{L_{2i}^{2}\left( x\right) }{4}-1}}-\sqrt
[2i]{\frac{L_{2i}\left( x\right) }{2}-\sqrt{\frac{L_{2i}^{2}\left(
x\right) }{4} -1}}=\left|x\right|,$\\
$\sqrt[2i]{\frac{L_{2i}\left(
x\right)}{2}+\sqrt{\frac{L_{2i}^{2}\left( x\right) }{4}-1}}+\sqrt
[2i]{\frac{L_{2i}\left( x\right) }{2}-\sqrt{\frac{L_{2i}^{2}\left(
x\right) }{4} -1}}=\sqrt{x^2+4}.$

 \item For all $x\in \mathbb{R}$\\ $\sqrt[2i]{\frac{
Q_{2i}\left( x\right) }{2}+\sqrt{\frac{Q_{2i}^{2}\left( x\right)
}{4}-1}}-\sqrt [2i]{\frac{Q_{2i}\left( x\right)
}{2}-\sqrt{\frac{Q_{2i}^{2}\left( x\right) }{4}
-1}}=2\, |x|,$\\
$\sqrt[2i]{\frac{
Q_{2i}\left( x\right) }{2}+\sqrt{\frac{Q_{2i}^{2}\left( x\right)
}{4}-1}}+\sqrt [2i]{\frac{Q_{2i}\left( x\right)
}{2}-\sqrt{\frac{Q_{2i}^{2}\left( x\right) }{4}
-1}}=2\, \sqrt{x^2+1}.$
\item For all $x\in {\mathbb{R}\setminus
(-\frac{2\sqrt2}{3},\frac{2\sqrt2}{3})}$ \\ $\sqrt[2i]{\frac{
f_{2i}\left( x\right) }{2}+\sqrt{\frac{f_{2i}^{2}\left( x\right)
}{4}-2^{2i}}}-\sqrt [2i]{\frac{f_{2i}\left( x\right)
}{2}-\sqrt{\frac{f_{2i}^{2}\left(
x\right) }{4} -2^{2i}}}=\sqrt{9\, x^2-8},$\\
$\sqrt[2i]{\frac{
f_{2i}\left( x\right) }{2}+\sqrt{\frac{f_{2i}^{2}\left( x\right)
}{4}-2^{2i}}}+\sqrt [2i]{\frac{f_{2i}\left( x\right)
}{2}-\sqrt{\frac{f_{2i}^{2}\left(
x\right) }{4} -2^{2i}}}=3\, |x|.$\\

\item For all $x\in {\mathbb{R}\setminus
(-1,1)}$ \\ $\sqrt[2i]{{ T_{2i}\left( x\right)
}+\sqrt{{T_{2i}^{2}\left( x\right) }-1}}-\sqrt[2i]{{ T_{2i}\left(
x\right) }+\sqrt{{T_{2i}^{2}\left( x\right) }-1}}=2\, \sqrt{x^2-1},$\\
$\sqrt[2i]{{ T_{2i}\left( x\right)
}+\sqrt{{T_{2i}^{2}\left( x\right) }-1}}+\sqrt[2i]{{ T_{2i}\left(
x\right) }+\sqrt{{T_{2i}^{2}\left( x\right) }-1}}=2\, |x|$.
\end{enumerate}
\end{cor}
%Setting $x=1$ in the first three relations of the above corollary, we
%get the following reduction formulae of some even radicals.
% \begin{cor}
%For every positive integer $i$, one has
%\begin{align*}
%&\sqrt[2i]{\frac{ L_{2i} }{2}+\sqrt{\frac{L_{2i}^{2} }{4}-1}}-\sqrt
%[2i]{\frac{L_{2i} }{2}-\sqrt{\frac{L_{2i}^{2} }{4} -1}}=1,\\
%&\sqrt[2i]{\frac{ L_{2i} }{2}+\sqrt{\frac{L_{2i}^{2} }{4}-1}}+\sqrt
%[2i]{\frac{L_{2i} }{2}-\sqrt{\frac{L_{2i}^{2} }{4} -1}}=\sqrt{5},\\
%&\sqrt[2i]{\frac{ Q_{2i} }{2}+\sqrt{\frac{Q_{2i}^{2} }{4}-1}}-\sqrt
%[2i]{\frac{Q_{2i} }{2}-\sqrt{\frac{Q_{2i}^{2} }{4} -1}}=2, \\
%&\sqrt[2i]{\frac{ Q_{2i} }{2}+\sqrt{\frac{Q_{2i}^{2} }{4}-1}}+\sqrt
%[2i]{\frac{Q_{2i} }{2}-\sqrt{\frac{Q_{2i}^{2} }{4} -1}}=2\, \sqrt{2}, \\
%&\sqrt[2i]{\frac{ f_{2i} }{2}+\sqrt{\frac{f_{2i}^{2}
%}{4}-2^{2i}}}-\sqrt [2i]{\frac{f_{2i} }{2}-\sqrt{\frac{f_{2i}^{2}
%}{4} -2^{2i}}}=1,\\
%&\sqrt[2i]{\frac{ f_{2i} }{2}+\sqrt{\frac{f_{2i}^{2}
%}{4}-2^{2i}}}+\sqrt [2i]{\frac{f_{2i} }{2}-\sqrt{\frac{f_{2i}^{2}
%}{4} -2^{2i}}}=3.
%\end{align*}
% \end{cor}
 In the following example, we give formulae of some specific even radicals based on Corollary \ref{evenrad1}.
\begin{exa}
The following identities are valid:
\begin{enumerate}
\item $\sqrt[4]{17+12\sqrt{2}}-\sqrt[4]{17-12\sqrt{2}}=
\\\\
\sqrt[4]{\frac{L_{4}(2)}{2}+\sqrt{\frac{(L_{4}(2))^{2}
}{4}-1}}-\sqrt [4]{\frac{L_{4}(2)}{2}-\sqrt{\frac{(L_{4}(2))^{2}
}{4}-1}}=2.$
\item
$\sqrt[4]{5201+1020\sqrt{26}}-\sqrt[4]{5201-1020\sqrt{26}}=
\\\\
\sqrt[4]{\frac{Q_{4}(5)}{2}+\sqrt{\frac{(Q_{4}(5))^{2}
}{4}-1}}-\sqrt [4]{\frac{Q_{4}(5)}{2}-\sqrt{\frac{(Q_{4}(5))^{2}
}{4}-1}}=10.$
\item
$\sqrt[6]{\frac{83448209}{2}+\frac{4010265}{2}\sqrt{433}}+\sqrt[6]{\frac{83448209}{2}-\frac{4010265}{2}\sqrt{433}}=
\\\\
\sqrt[6]{\frac{f_{6}(7)}{2}+\sqrt{\frac{(f_{6}(7))^{2}
}{4}-2^6}}+\sqrt [6]{\frac{f_{6}(7)}{2}-\sqrt{\frac{(f_{6}(7))^{2}
}{4}-2^6}}=21.$
\item
$\sqrt[8]{708158977+408855776 \sqrt{3}}+\sqrt[8]{708158977-408855776 \sqrt{3}}=\\\\
\sqrt[8]{T_{8}(7)+\sqrt{(T_{8}(7))^{2}-1}}+\sqrt
[8]{T_{8}(7)-\sqrt{(T_{8}(7))^{2}-1}}=14.$

\end{enumerate}
\end{exa}
 \section {Some other Radicals formulae}
  This section is devoted to establishing other two reduction formulae of some radicals. First, the following lemma is needed in the sequel.
   \begin{lem}
   \label{lem3}
    For every positive integer $j$ and every $x\in
\mathbb{R}^*$, one has
\bq
\psi^{a,b}_{j}\left(\frac{1}{a}(x-\tfrac{b}{x})\right)
=\begin{cases}
x^j-\dd\frac{b^j}{x^j},& j\ \text{odd},\\[0.3cm]
x^j+\dd\frac{b^j}{x^j},& j\ \text{even}.
\end{cases}
\eq
\end{lem}
\begin{proof}
Binet's formula \eqref{BinetLuc} implies that
\begin{align*}
&\hspace*{120pt}\psi^{a,b}_{j}\left(\frac{1}{a}(x-\tfrac{b}{x})\right)\\
&=\frac{1}{2^j}\left\{\left(\left(x-\tfrac{b}{x}\right)+\sqrt{\left(x-\tfrac{b}{x}\right)^2+4b}\right)^j+
\left((x-\tfrac{b}{x})-\sqrt{(x-\tfrac{b}{x})^2+4b}\right)^j\right\}\\
&=\frac{1}{2^j}\left\{\left((x-\tfrac{b}{x})+\sqrt{(x+\tfrac{b}{x})^2}\right)^j+\left((x-\tfrac{b}{x})
-\sqrt{(x+\tfrac{b}{x})^2}\right)^j\right\}\\
&=\frac{1}{2^j}\left\{\left(\left(x-\tfrac{b}{x}\right)+\left|{x+\tfrac{b}{x}}\right|\right)^j+\left((x-\tfrac{b}{x})
-\left|{x+\tfrac{b}{x}}\right|\right)^j\right\}\\
&=\frac{1}{2^j}\left\{(2\, x)^j+\left(\tfrac{-2\, b}{x}\right)^j\right\}=\begin{cases}
x^j-\frac{b^j}{x^j},& j\ \text{odd},\\
x^j+\frac{b^j}{x^j},& j\ \text{even}.
\end{cases}
\end{align*}
The proof is now complete.
\end{proof}
     \begin{thm}
   \label{thm7}
Let $j$ be any positive odd integer. If $b\in \mathbb{R}$, then for every $x\in
\mathbb{R}^*$, the following identity holds:
    \bq
    \begin{split}
    \label{oddthm71}
    &\sqrt[j]{\frac{\psi^{a,b}_{j}(\frac{1}{a}(x-\frac{b}{x}))}{2}+
    \sqrt{\dd\frac{\left(\psi^{a,b}_{j}(\frac{1}{a}(x-\frac{b}{x}))\right)^2}{4}+b^j}}=\\
     &\begin{cases}
     x,&x>0,\ b>0,\\
     \frac{-b}{x},&x<0,\ b>0,\\
     x,&x\in[-\sqrt{-b},0)\cup [\sqrt{-b},\infty),\, b<0,\\
     \frac{-b}{x},&x\in(-\infty,-\sqrt{-b})\cup (0,\sqrt{-b}),\, b<0.
    \end{cases}
    \end{split}
   \eq
\end{thm}
\begin{proof}
From Lemma \ref{lem3}, we have
$$\sqrt[j]{\frac{\psi^{a,b}_{j}(\frac{1}{a}(x-\frac{b}{x}))}{2}+
\sqrt{\dd\frac{\left(\psi^{a,b}_{j}(\frac{1}{a}(x-\frac{b}{x}))\right)^2}{4}+b^j}}
=\sqrt[j]{\frac{(x^j-\frac{b^j}{x^j})}{2}+\sqrt{\dd\frac{(x^j-\frac{b^j}{x^j})^2}{4}+b^j}}$$
$$=\sqrt[j]{\frac{(x^j-\frac{b^j}{x^j})}{2}+\sqrt{\dd\frac{(x^j+\frac{b^j}{x^j})^2}{4}}}
=\sqrt[j]{\frac{(x^j-\frac{b^j}{x^j})}{2}+\frac{|x^j+\frac{b^j}{x^j}|}{2}},$$
and consequently, relation \eqref{oddthm71} can be obtained.
    \end{proof}
The following results are direct consequences of Theorem \ref{thm7}.
\begin{cor}
\label{cor14Mod}
 If $j$ is an odd positive integer, then the following identities are true for Lucas,
Pell-Lucas, Fermat-Lucas, and Chebyshev polynomials of first kind:
   \begingroup
\allowdisplaybreaks
\begin{align*}
&\sqrt[j]{\frac{L_{j}(x-\frac{1}{x})}{2}+
\sqrt{\dd\frac{\left(L_{j}(x-\frac{1}{x})\right)^2}{4}+1}}
=\begin{cases}
     x,&x>0,\\
     \frac{-1}{x},&x<0,
    \end{cases}\\
 &\sqrt[j]{\frac{Q_{j}(\frac{1}{2}(x-\frac{1}{x}))}{2}+
\sqrt{\dd\frac{\left(Q_{j}(\frac{1}{2}(x-\frac{1}{x}))\right)^2}{4}+1}}
=\begin{cases}
     x,&x>0,\\
     \frac{-1}{x},&x<0,
    \end{cases}\\
    &\sqrt[j]{\frac{f_{j}(\frac{1}{3}(x+\frac{2}{x}))}{2}+
\sqrt{\dd\frac{\left(f_{j}(\frac{1}{3}(x+\frac{2}{x}))\right)^2}{4}-2^j}}
=\begin{cases}
     x,&x\in [\sqrt{2},\infty)\cup[-\sqrt{2},0),\\
     \frac{2}{x},&x\in (-\infty,-\sqrt{2})\cup(0,\sqrt{2}),
    \end{cases}\\
    &\sqrt[j]{T_{j}(\tfrac{1}{2}(x+\tfrac{1}{x}))+\sqrt{\left(T_{j}(\tfrac{1}{2}(x+\tfrac{1}{x}))\right)^2}-1}=\begin{cases}
     x,&x\in [1,\infty)\cup[-1,0),\\
     \frac{1}{x},&x\in (-\infty,-1)\cup(0,1).
    \end{cases}
\end{align*}
\endgroup
\end{cor}
\begin{exa}
As an application of Corollary \ref{cor14Mod}, which is a corollary of Theorem \ref{thm7}, we readily get (for $j=5$ and $x=\frac{-1}{3}$), the following odd reduction formula
\[\sqrt[5]{\frac{-1889569}{486}+\sqrt{\frac{3570463447489}{236196}}}=
\sqrt[5]{\frac{f_{5}(\frac{-19}{9})}{2}+
\sqrt{\frac{(f_{5}(\frac{-19}{9}))^2}{4}-32}}=-\frac{1}{3}.\]
\end{exa}
 The following theorem exhibits the counterpart result of Theorem \ref{thm7} for even radicals.
     \begin{thm}
   \label{thm8}
Let $j$ be any positive even integer. Then for every nonzero real
number $x$, the following identity holds:
    \bq
    \label{oddre2}
    \sqrt[j]{\frac{\psi^{a,b}_{j}(\frac{1}{a}(x-\frac{b}{x}))}{2}
    +\sqrt{\dd\frac{\left(\psi^{a,b}_{j}(\frac{1}{a}(x-\frac{b}{x}))\right)^2}{4}-b^j}}
    =\begin{cases}
    |x|,&x\in{\mathbb{R}\setminus{(-\sqrt{|b|},\sqrt{|b|})}},\\
    |\frac{b}{x}|,&x\in(-\sqrt{|b|},\sqrt{|b|}).
    \end{cases}
    \eq
\end{thm}
\begin{exa} The dedication for the paper \cite{berndt1997ramanujan}
reads: "In memory of Ramanujan on the
$$\sqrt[6]{32(\tfrac{146410001}{48400})^3-6(\tfrac{146410001}{48400})+
\sqrt{(32(\tfrac{146410001}{48400})^3-6(\tfrac{146410001}{48400}))^2-1}}
\quad  th $$ anniversary of his birth.\\
This radical was calculated by Osler in \cite{osler2001cardan} using
Cardan polynomials. In the following few lines, we will show that it
can be evaluated with the aid of Theorem \ref{thm8}. In fact, if we
note that
 \[32\left(\tfrac{146410001}{48400}\right)^3-6\left(\tfrac{146410001}{48400}\right)=
\frac{1}{2}\psi^{1,-1}_{3}\ \left((110)^2+\tfrac{1}{(110)^2}\right),\]
  then we have
  \begin{align*}
  &\sqrt[6]{32(\tfrac{146410001}{48400})^3-6(\tfrac{146410001}{48400})+
\sqrt{(32(\tfrac{146410001}{48400})^3-6(\tfrac{146410001}{48400}))^2-1}}=\\
&\sqrt[6]{\tfrac{\psi^{1,-1}_{3}\left((110)^2+\tfrac{1}{(110)^2}\right)}{2}+
\sqrt{\tfrac{\left(\psi^{1,-1}_{3}\left((110)^2+\tfrac{1}{(110)^2}\right)\right)^2}{4}-1}}=
\sqrt{(110)^2}=110.
  \end{align*}
\end{exa}
%\begin{exa}
%As two examples of Theorems \ref{thm7} and \ref{thm8}, we
%give the following two results:
%\begin{enumerate}
%\item
%$\sqrt[4]{\frac{195313}{625}+\sqrt{\frac{38146777344}{390625}}}=
%\sqrt[4]{\frac{Q_{4}(\frac{12}{5})}{2}+
%\sqrt{\frac{(Q_{4}(\frac{12}{5}))^2}{4}-1}}=5 ,$
%\item $\sqrt[5]{\frac{-1889569}{486}+\sqrt{\frac{3570463447489}{236196}}}=
%\sqrt[5]{\frac{f_{5}(\frac{-19}{9})}{2}+
%\sqrt{\frac{(f_{5}(\frac{-19}{9}))^2}{4}-32}}=-\frac{1}{3}.$
%\end{enumerate}
%\end{exa}
\section{Conclusions}
In this paper, we established several new connection formulae between some classes of polynomials generalizing the celebrated Fibonacci and Lucas polynomials. A few of these formulae generalize respective known ones. Two applications of the derived connection formulae were presented. A number of new expressions between the celebrated
Fibonacci, Pell, Fermat, Lucas, Pell-Lucas, and Fermat Lucas numbers were deduced. As a very important utilization of the two generalized Fibonacci and Lucas polynomials, new reduction formulae of certain odd and even radicals were given. We have deduced as special cases several reduction formulae of certain radicals involving Lucas, Pell Lucas, Fermat Lucas, Chebyshev and Dickson polynomials of the first and second kinds. Numerous examples are given to apply reduction of radicals. To the best of our knowledge, most of the formulae in this paper are new and we do believe that they are important and helpful.
%%%%%%%%%%%%%%%%%%%%%%%%%%%%%%%%%%%%%%%%%%%%%%%%%%%%%%%%%%%%%%%%%%%%%%%%%%%%%%%%%%%%%%%%%%%%%%%%%%%%%%%%%%%%%%%%%

%%%%%%%%%%%%%%%%%%%%%%%%%%%%%%%%%%%%%%%%%%%%%%%%%%%%%%%%%%%%%%%%%%%%%%%%%%%%%%%%%%%%%%%%%%%%%%%%%%%%%%%%%%%%%%%

\end{document}